\documentclass[a4wide,11pt,reqno]{article} 
\usepackage{amsmath,latexsym,amsfonts,amssymb}
\usepackage{mathrsfs}
\usepackage{fullpage}
\usepackage{enumerate}
\usepackage{amssymb, amsmath, amsfonts}
\usepackage{mathrsfs}
\usepackage{epsfig,amsbsy,amsthm} 
\usepackage{float,epsfig}
\usepackage{fixmath}
\usepackage{graphics}
\usepackage{pgfpages}
\usepackage{caption}
\usepackage{subcaption}
\usepackage{appendix}
\numberwithin{equation}{section}  
\usepackage{hyperref} 
\usepackage{graphicx}
\usepackage{pst-all}
\usepackage{calligra}
\usepackage{color}
\usepackage{tikz}
\DeclareMathAlphabet{\mathpzc}{OT1}{pzc}{m}{it}
\DeclareMathAlphabet{\mathcalligra}{T1}{calligra}{m}{n}
\begin{document}
\newtheorem{theorem}{\bf Theorem}[section]
\newtheorem{proposition}[theorem]{\bf Proposition}
\newtheorem{definition}{\bf Definition}[section]
\newtheorem{corollary}[theorem]{\bf Corollary}
\newtheorem{exam}[theorem]{\bf Example}
\newtheorem{remark}[theorem]{\bf Remark}
\newtheorem{lemma}[theorem]{\bf Lemma}
\newtheorem{assum}[theorem]{\bf Assumption}

\newcommand{\von}{\vskip 1ex}
\newcommand{\vone}{\vskip 2ex}
\newcommand{\vtwo}{\vskip 4ex}
\newcommand{\ds}{\displaystyle}
\def \noin{\noindent}
\newcommand{\be}{\begin{equation}}
\newcommand{\ee}{\end{equation}}
\newcommand{\beno}{\begin{equation*}}
\newcommand{\eeno}{\end{equation*}}
\newcommand{\ba}{\begin{align}}
\newcommand{\ea}{\end{align}}
\newcommand{\bano}{\begin{align*}}
\newcommand{\eano}{\end{align*}}
\newcommand{\bea}{\begin{eqnarray}}
\newcommand{\eea}{\end{eqnarray}}
\newcommand{\beano}{\begin{eqnarray*}}
\newcommand{\eeano}{\end{eqnarray*}}
\def \noin{\noindent}
\def \R{{\mathbb R}}
\def \V{{\mathbb V}}
\def \S{{\mathbb S}}
\def \N{{\mathbb N}}
\def \Z{{\mathbb Z}}
\def \Mc{{\mathcal M}}
\def \Cc{{\mathcal C}}
\def \Rc{{\mathcal R}}
\def \Ec{{\mathcal E}}
\def \Gc{{\mathcal G}}
\def \Tc{{\mathcal T}}
\def \Qc{{\mathcal Q}}
\def \Ic{{\mathcal I}}
\def \Pc{{\mathcal P}}
\def \Oc{{\mathcal O}}
\def \Uc{{\mathcal U}}
\def \Yc{{\mathcal Y}}
\def \Ac{{\mathcal A}}
\def \Bc{{\mathcal B}}
\def \k{\mathpzc{k}}
\def \Rp{\mathpzc{R}}
\def \Os{\mathscr{O}}
\def \Js{\mathscr{J}}
\def \Es{\mathscr{E}}
\def \Qs{\mathscr{Q}}
\def \Ss{\mathscr{S}}
\def \Cs{\mathscr{C}}
\def \Ds{\mathscr{D}}
\def \Ms{\mathscr{M}}
\def \Ts{\mathscr{T}}
\def \LL{L^{\infty}(L^{2}(\Omega))}
\def \LH{L^{2}(0,T;H^{1}(\Omega))}
\def \B {\mathrm{BDF}}
\def \el {\mathrm{el}}
\def \re {\mathrm{re}}
\def \e {\mathrm{e}}
\def \div {\mathrm{div}}
\def \CN {\mathrm{CN}}
\def \Rs   {\mathbf{R}_{{\mathrm es}}}
\def \Rb {\mathbf{R}}
\def \Jb {\mathbf{J}}
\def  \apos {\emph{a posteriori~}}

\title{A  nonconforming finite element method for an elliptic optimal control problem with constraint on the gradient}

\author{ Kamana Porwal
\thanks{ Department of Mathematics, Indian Institute of Technology Delhi,
New Delhi- 110016, India (\tt{kamana@maths.iitd.ac.in})}
 ~~and ~~
 Pratibha Shakya\thanks{ (\tt{shakya.pratibha10@gmail.com})}\\
}
\date{}
\maketitle
\textbf{ Abstract.}{\small{ This article is concerned with the nonconforming finite element method for distributed elliptic optimal control problems with pointwise constraints on the control and gradient of the state variable.  
We reduce the minimization problem into a pure state constraint minimization problem. In this case, the solution of the minimization problem can be characterized as fourth-order elliptic variational inequalities of the first kind. To discretize the control problem we have used the bubble enriched Morley finite element method. To ensure the existence of the solution to discrete problems three bubble functions corresponding to the mean of the edge are added to the discrete space. We derive the error in the state variable in $H^2$-type energy norm. Numerical results are presented to illustrate our analytical findings.
\\



\section{Introduction}
PDE-constrained optimization problem consisting of the state which is a solution of a partial differential equation and constraints on the control and/or state. These optimization problems have many application backgrounds in science, engineering, and many real-life problems. The literature is quite vast,  we refer to the reader \cite{Lions:1971:OCP, fredi10} for its theory and applications.  

 There is a wide range of articles available for the numerical approximation of control constrained optimal control problems.  The finite element method was discussed for the numerical approximation of elliptic optimal control problems in early papers by Falk \cite{Falk:1973:OCP}, Geveci \cite{Geveci:1979:OCP}, Arnautu and Neittaanm\"{a}ki \cite{arnautu}, they proved the error estimates in the $L^2$-norm. The authors of \cite{Arada:casas:2002} have derived the error estimate for the control in the $L^{\infty}$ and $L^2$-norms for the semi-linear elliptic control problem.  Casas and Tr\"{o}ltzsch in \cite{casas:trolzsch:2005} have derived the error estimates because every nonsingular solution can be approximated by a sequence $(u_h)_h$ of discrete controls. Meyer  and R\"{o}sch in \cite{meyer:rosch:Linfty} have proved the optimal order error estimate for the control in $L^{\infty}$ norm for  the two dimensional bounded domains with $\mathcal{C}^{1,1}$-boundary. They have used the piecewise linear polynomials for the discretization of the control variable. 
 Moreover, the numerical approximation of the elliptic optimal control problems with control from the measure space can be found in \cite{casas:2012:sparse,clason:2011}. Recently, error estimates for the state and control for the elliptic optimal control problem with nonsmooth data presented in \cite{pshakya:2019}.
 
In \cite{casas:mateos:troltzsch:2005},  Casas {\it et al.} present numerical analysis for Neumann boundary control of semilinear elliptic equations and derived the error estimate for the control in $L^2(\Gamma)$-norm, where they have used the piecewise constant polynomials for the approximation of the control variable. The authors of \cite{casas:mateos:2008} have analyzed the error estimates for the control using various discretizations of the control variable like piecewise linear and continuous and variational discretization.
 Subsequently, Casas and Raymond \cite{casas:raymond:2006} have studied semilinear elliptic boundary control problems with pointwise constraints on the control. The authors have used continuous piecewise linear finite elements for the approximation of the state as well as a control variable and the related error estimates are derived.
 In \cite{may:rannacher:vexler:2007}, May, Rannacher and Vexler studied Dirichlet boundary control problem without control constraints. They derived error estimates in the weaker norm for the two-dimensional convex polygon.
  Error analysis for general two- and three-dimensional curved domains is presented by Deckelnick, G\"{u}nther and Hinze in \cite{deckelnick:gunther:hinze:2009}.
  
  Later on, the authors of \cite{may13} have considered an unconstrained  Dirichlet boundary control problem on convex polygonal domains and presented the optimal error estimates for the state and the control variables.

State constrained optimal control problems have many application background. For an overview concerning numerical approximatimation of state contrained elliptic optimal control problems, we refer to \cite{Casas:1993:StateConst,casas:finite,casas:M,casas:14,DH:2007:StateConsrt:1,DH:2007:StateCtrl,Meyer:2008:StateConst}.
The existence and uniqueness of the pointwise state constrained optimization problem was discussed in \cite{Casas:1993:StateConst}.
For an elliptic optimal control problem with finitely many state constraints, Casas in \cite{casas:finite}  has derived error estimates for the control in $L^{\infty}$-norm. Further, error estimates for Lagrange multipliers associated with the state constraints, state, and co-state variables were also obtained in \cite{casas:finite}.  Using less regularity,  Casas and Mateos in \cite{casas:M} have addressed the convergence results for the state of semi-linear distributed and boundary control problems. In \cite{DH:2007:StateConsrt:1,DH:2007:StateCtrl} Deckelnick and Hinze  have presented the error estimates for  two- and three dimensional spatial domain. In \cite{Meyer:2008:StateConst} Meyer addressed a fully discrete strategy to derive the error estimates of state and control constrained elliptic optimal control problem. They derive the error estimates for control in the $L^2$-norm and state in the $H^1$-norm. Later, the authors of \cite{casas:14} have used new regularity results to improve the error estimates for the state and control variables. 
We refer to \cite{LYYG:2010:IntState,ZY:2015:LGS,lixin17} for elliptic optimal control problems with integral state constraints.

Optimal control problems with the gradient constraint on the state play an important role in many practical applications, like,  solid mechanics, large temperature gradients during cooling or heating of any object may lead to its destruction, in material science to avoid large material stresses.
The bounds on the gradient of the state variable lead to the low regularity of the state variable and in the optimization, the adjoint variable admits low regularity. A great number of researches for gradient state constrained optimal control problems. The existence and uniqueness of the solution of optimal control problems governed by semilinear elliptic and parabolic state equations were addressed in \cite{casas1993,casas2002,casasmateos}. 
 Later, Griesse and Kunisch in  \cite{griesse2009} have provided a semi-smooth Newton method and regularized active set method for the solution of an elliptic partial differential equation with the constraint on the gradient of the state. The authors of \cite{schila11} were discussed barrier methods for the gradient constraint optimal control problems, governed by partial differential equations.
 For the barrier parameter, a posteriori error estimate was derived in \cite{wollner10}.
   Recently, the authors of \cite{herzog16} have obtained preconditioned solutions of the gradient constrained control problem. 
Deckelnick {\it et al.} in \cite{deckelnick09} have proposed a tailored finite element approximation to the minimization problem. Therein they have used a sequence of functionals to approximate the cost functional, whereas the sequence of functionals is obtained by the use of a discrete state equation.
The variational discretization is used to discretize the control problem, where the lowest order Raviart-Thomas element is used to approximate the state equation. They derived error estimates for the state and control variables. The authors of \cite{wollner11} have derived a priori error estimates for the state and control variables for the elliptic optimal control problem with the constraint on the gradient of the state variable. 
G\"{u}nther and Hinze \cite{gunther11} have discussed error bounds for the state and control variables in two and three-dimensional 
problem. For the discretization of the state variable, they use piecewise linear finite elements and for the control variable, they used the variational discretization technique and piecewise constant control and compare both the results. 
The authors of \cite{casas1993} have discussed the existence and uniqueness results for the solution of the optimization problem governed by a semi-linear state equation with pointwise constraints on the gradient. Recently, Wollner in \cite{wollner12} has discussed the existence of a solution on the non-smooth polygonal and polyhedral domain and derived the optimality conditions. 
Concerning adaptive discretization methods, we refer to  \cite{hintermuller2012}. 
The authors of \cite{hintermuller2009} have used Moreau-Yosida based framework for the minimization problem. They have considered the minimization problem governed by partial differential equations with pointwise control on the control, state, and gradient of the state variable. 

 The main intent of this article is to discuss the asymptotic behavior of the state variable of a nonconforming finite element method.  For this, we have used the new approach discussed in  \cite{BSZ:2013:C0IP,brenner:sung:2017,LYG:2009:StateConst}. In this approach, the optimal control problem reduces to the new minimization problem involving only the state variable. The solution of the resulting minimization problem can be characterized by the solution of a fourth-order variational inequality and hence we obtained the convergence behavior in $H^2$ type norm.  The bubble-enriched Morley finite element is used for the discretization of the problem.

 In this paper,  we consider the following distributed elliptic optimal control problem 
\begin{eqnarray}
\text{Find}\;\; ({y}^*,{u}^*)= \displaystyle argmin_{(y,u)\in {K}} \Big[\frac{1}{2}\int_{\Omega}(y-y_d)^2\,dx+\frac{\beta}{2}\int_{\Omega}u^2\,dx\Big],\label{intro:functional}
\end{eqnarray}
subject to the state equation
\begin{eqnarray}
\int_{\Omega}\nabla y\cdot \nabla w\,dx=\int_{\Omega} uw\,dx\;\;\;\;\forall w\in H^1_0(\Omega),\label{intro:state}
\end{eqnarray}
 the gradient  state and control constraints
\begin{align}
&|\nabla y|\leq 1\;\;\;\;\;\;\;\;\;\,\text{in}\;\;\Omega,\label{intro:state:cons}\\
&u_a\leq u\leq u_b\;\;\;\;\text{in}\;\;\Omega.\label{intro:control:cons}
\end{align}
where $\Omega\subset\mathbb{R}^2$ with smooth boundary $\partial\Omega$, $(y,u)$ belongs to ${K}\subset H^1_0(\Omega)\times L^2(\Omega)$. The given desired state $y_d\in L^2(\Omega)$ and  $\beta$ be a positive constant. Assume that the given functions 
  $u_a$ and $u_b$  satisfy $(i)\; u_a,u_b\in W^{1,\infty}(\Omega)$ and $(ii)\; u_a<u_b$ on $\bar{\Omega}$.
  
 We plan our exposition as follows:  Section $2$ is devoted to the existence, uniqueness, and regularity results of the control problem. In Section $3$, we introduce the discrete problem and properties of interpolation and enriching operator. The asymptotic behavior of the solution has been established in Section $4$. Lastly, the numerical experiment is performed to illustrate our theoretical behavior.

\section{The  control problem}
 This section is devoted to the existence, uniqueness, and regularity of the control problem.
 
 Let $\Omega$ be a bounded convex domain with smooth boundary $\partial\Omega$. We use the notation $W^{m,p}(\Omega)$ for Sobolev spaces on $\Omega$ with norm $\|\cdot\|_{W^{m,p}(\Omega)}$ and seminorm $|\cdot|_{W^{m,p}(\Omega)}$. We set  $W_0^{m,p}\Omega)=\{v\in W^{m,p}(\Omega):\,v|_{\partial\Omega}=0\}.$ In addition $C$ denotes a positive generic constant independent of the mesh parameter.

 If the domain $\Omega$ is convex, then the regularity theory of the elliptic equation says that there exists a unique solution $y\in H^2(\Omega)\cap H^1_0(\Omega)$ of (\ref{intro:state}).
Using $u=-\Delta y$,  the minimization problem (\ref{intro:functional})-(\ref{intro:control:cons}) can be rewritten as follows
 \begin{align}
 \text{Find}\; {y}^*&= \displaystyle argmin_{y\in \mathcal{K}} \Big[\frac{1}{2}\int_{\Omega}(y-y_d)^2\,dx+\frac{\beta}{2}\int_{\Omega}(-\Delta y) ^2\,dx\Big],\nonumber\\
 &=argmin_{y\in \mathcal{K}} \Big[\frac{1}{2}\mathcal{A}(y,y)-(y_d,y)\Big],\label{intro:func:modi}
 \end{align}
 where 
 \begin{align}
 \mathcal{K}=\{v\in W^{2,r}(\Omega)\cap H^1_0(\Omega):\,|\nabla v|\leq 1\;\;\text{and}\;\;u_a\leq -\Delta v\leq u_b\;\;\text{in}\;\Omega\,\,\, \text{with}\;\; r>2\},\label{intro:k}
 \end{align}
 and the bilinear form $\mathcal{A}(\cdot,\cdot)$ be defined by
 \begin{align*}
 \mathcal{A}(v,w)&=\beta\int_{\Omega}(\Delta v)(\Delta w)\,dx+\int_{\Omega} vw\,dx,\notag\\
 &=\beta\int_{\Omega}\sum_{i,j=1}^{2}\Big(\frac{\partial^2 v}{\partial x_{i}\partial x_{j}}\Big)\Big(\frac{\partial^2 w}{\partial x_{i}\partial x_{j}}\Big)dx+\int_{\Omega} vw\,dx
  \end{align*}
 and $(v,w)=\int_{\Omega}vw\,dx.$
 \noindent
To ensure that there exists a solution we assume the following conditions holds:\\
There exists $y\in W^{2,r}(\Omega)\cap H^1_0(\Omega)$ satisfies (i) $|\nabla y|< 1$ in \;$\Omega$ and (ii) $u_a\leq (-\Delta y)\leq u_b$.
\begin{remark}
The above conditions are known as the Slater condition. These conditions play a crucial role in state-constrained optimal control problems. 
  \end{remark}
 \noindent 
   The bilinear form $\mathcal{A}(\cdot,\cdot)$ is symmetric, bounded and coercive on $H^2(\Omega)\cap H^1_0(\Omega)$.  The  closed convex set $\mathcal{K}$ is nonempty.  From the classical theory (cf. \cite{kinderlehrer}), the unique solution ${y}^*\in W^{2,r}(\Omega)\cap H^1_0(\Omega)$ of (\ref{intro:func:modi})-(\ref{intro:k})   is  characterized by the variational inequality
  \begin{align}
  \mathcal{A}(y,y-y^*)\geq (y_d, y-y^*)\label{variational:ineq}.
  \end{align}
  From the Lagrangian approach, we have the following existence results.  
\begin{theorem}
We have the following Karush-Kuhn-Tucker conditions for (\ref{variational:ineq}).
\begin{align}
\mathcal{A}({y}^*,v)=\int_{\Omega}{y}_dv\,dx+\int_{\Omega}\lambda(-\Delta v)\,dx+\int_{\Omega}\nabla v\cdot\nabla {y}^*\,d\mu\label{kkt}
\end{align}
together with the complementary condition
\begin{eqnarray}
&&\lambda\geq 0\;\;\;\;\;\text{if}\;\;-\Delta{y}^*=u_a,\label{lambda:1}\\
&&\lambda\leq 0\;\;\;\;\;\text{if}\;\;-\Delta{y}^*=u_b,\label{lambda:2}\\
&&\lambda=0\;\;\;\;\;\text{otherwise},\label{lambda:3}\\
&&\langle |\nabla {y}^*|^2,\mu\rangle_{\mathcal{C},\mathcal{C}^*}\leq\langle \nabla v\nabla {y}^*,\mu\rangle>_{\mathcal{C},\mathcal{C}^*} \;\;\;\;\;\forall\, \nabla v\in\mathcal{C}(\bar{\Omega},\mathbb{R}^2),\;\; |\nabla v|\leq 1.\label{comple:1}
\end{eqnarray}
\end{theorem}
\noindent
The adjoint equation is given by: Find $p^*\in L^{s}(\Omega)\;(s<2)$ satisfy $\frac{1}{s}+\frac{1}{s'}=1$  such that
\begin{eqnarray}
\int_{\Omega}p^*(-\Delta v)\,dx=\int_{\Omega}({y}^*-y_d)v\,dx-\int_{\Omega}\nabla{y}^*\cdot\nabla v\,d\mu\,\,\forall v\in W^{2,s'}(\Omega)\cap H^1_0(\Omega).\label{adjoint}
\end{eqnarray}
\noindent
The following regularity result for the adjoint state ${p}^*$ is taken from \cite{wollner11}.
\begin{lemma}\label{lemma:1}
For two dimensional domain, there exists constants $\gamma,\gamma'>0$ such that $\gamma+\gamma'\geq 1-\frac{2}{r}$, and 
the solution of the adjoint equation ${p}^*\in W^{\gamma,r'}(\Omega)\subset W^{\gamma,s}(\Omega)$.
\end{lemma}

\noindent
An use of  (\ref{kkt}) and (\ref{adjoint}) leads to
\begin{eqnarray}
\lambda=p^*+\beta(-\Delta {y}^*),\label{lambda}
\end{eqnarray}
An application of the complementary conditions (\ref{lambda:1})-(\ref{lambda:3}) gives
\begin{eqnarray}
-\Delta {y}^*=P_{[u_a,u_b]}\Big(-p^*/\beta\Big),
\end{eqnarray}
where  $P_{[u_a,u_b]}$ denote the orthogonal projection from $L^2(\Omega)$ onto   $\{v\in L^2(\Omega):\;u_a\leq v\leq u_b\}$ defined by
\begin{eqnarray*}
P_{[u_a,u_b]}(v):=\max\{u_a,\min\{u_b,v\}\}.
\end{eqnarray*}
\begin{corollary}\label{cor:1}
Assume that  $y^*$ satisfies
\begin{eqnarray*}
\begin{cases}
-\Delta {y}^*=P_{[u_a,u_b]}(-{p}^*/\beta)\;\;\text{in}\;\Omega,\\
{y}^*=0\;\;\text{on}\;\;\partial\Omega.
\end{cases}
\end{eqnarray*}
Then,  ${y}^*\in H^{2+\gamma}(\Omega)$.
\end{corollary}
\noindent
An application of   (\ref{lambda}), Lemma \ref{lemma:1} and Corollary \ref{cor:1}  gives $\lambda\in W^{\gamma,s}(\Omega)$ with  $s<2$.
Use of the regularity of the adjoint variable, regularity of $\lambda$ and integration by parts, will give $\mu\in W^{-\gamma,s}(\Omega)$.
\section{Finite element discretization}
 This section is devoted to the finite element approximation of the minimization problem.
 
 Let $\mathcal{T}_h$ be a quasi-uniform  triangulation of $\Omega$, $T$ be a triangle in $\mathcal{T}_h$, $\mathcal{V}_h$ be the set of the vertices of $\mathcal{T}_h$,  $\mathcal{E}_h$ be the set of the edges of $\mathcal{T}_h$, $\mathcal{E}_T$ be the set of three edges of $T$,
 $h_T$ be the diameter of $T$, $h$ be the  $\max_{T\in\mathcal{T}_h} diam(T)$ and $\Delta_h$ be the piece-wise  (element-wise) Laplacian operator.
 
\noindent
For each triangle $T\in\mathcal{T}_h$, let  $b_{T1},b_{T2}\;\; \text{and}\;\; b_{T3}$  denotes the bubble functions  corresponding to the mean value of the function at mid point of the edges.
Construct  
\begin{eqnarray}
W_h=\{v_h\in L^2(\Omega): v_h|_{T}\in span\{b_{T1},b_{T2},b_{T3}\}\;\;\forall T\in\mathcal{T}_h\}
\end{eqnarray} 
\noindent
Let $V_M$  denote the Morley finite element \cite{Morley:1968} space is defined as
\begin{align*}
V_M=&\{v_h\in L^2(\Omega):\,v_h|_{T}\in \mathbb{P}_2(T),\; v_h\,\text{is continuous at the vertices and the normal derivative of} \,v_h\\\,&\;\;\text{is continuous  at the midpoint of the edges,}\,\text{the degree of freedom of} \,v_h\, \text{vanish on} \,\partial\Omega\},  
\end{align*}
The discrete finite element space $V_h$ \cite{nilssen} be defined by
\begin{eqnarray*}
V_h=V_{M}\oplus W_h.
\end{eqnarray*}
The discrete form of the convex set $\mathcal{K}$ be defined as
\begin{eqnarray}
\mathcal{K}_h:=\{v_h\in V_h: |I_h\nabla v_h|\leq 1\;\;\text{and}\;\;Q_h u_a\leq Q_{h}(-\Delta_h v_h)\leq Q_h u_b\},\label{kh}
\end{eqnarray}
where $Q_h$ denote the orthogonal projection from $L^2(\Omega)$ onto the space of piecewise constant functions. The projection $I_h$ be defined by
\begin{eqnarray*}
I_h(\nabla v_h)&:=&\frac{1}{|T|}\int_{T}\nabla v_h\,dx,\\
\end{eqnarray*}

\noindent
The finite element approximation of the minimization problem (\ref{intro:func:modi}) is defined as follows: Find 
\begin{eqnarray}
{y}^*_h=argmin_{y_h\in \mathcal{K}_h}\Big[\frac{1}{2}\mathcal{A}_h(y_h,y_h)-(y_d,y_h)\Big], \label{dis:prob}
\end{eqnarray}
where 
\begin{eqnarray}
\mathcal{A}_h(v_h,w_h)=\beta\sum_{T\in\mathcal{T}_h}\int_{T}D^2 v_h:D^2 w_h\,dx+\int_{\Omega}v_hw_h\,dx
\end{eqnarray}
\noindent
 and the term 
$D^2v_h:D^2w_h=\displaystyle\sum_{i,j=1}^{2}\Big(\frac{\partial^2 v_h}{\partial x_{i}\partial x_{j}}\Big)\Big(\frac{\partial^2 w_h}{\partial x_{i}\partial x_{j}}\Big)$. 
\noindent
 The mesh dependent norm $\|\cdot\|_h$ defined by
 \begin{eqnarray}
 \|v_h\|_h^2=\mathcal{A}_h(v_h,v_h)=\beta\sum_{T\in \mathcal{T}_h}|v_h|_{H^2(T)}^2+\|v_h\|_{L^2(\Omega)}^2,\label{def:discrete:norm}
 \end{eqnarray}
 and we have the following estimates
 \begin{eqnarray}
 |\mathcal{A}_h(v_h,w_h)|&\leq& C_1\|v_h\|_h\|w_h\|_h\;\;\;\;\;\forall v_h,w_h\in [H^2(\Omega)\cap H^1_0(\Omega)]+V_h,\label{continuous}\\
 \mathcal{A}_h(v_h,v_h)&\geq &C_2\|v_h\|_h^2\;\;\;\;\;\;\forall v_h\in V_h.\label{coercive}
 \end{eqnarray}
 \textbf{The interpolation operator $\pi_h.$} The interpolation operator $\pi_h:H^2(\Omega)\cap H^1_0(\Omega)\rightarrow V_h$ is defined by
 \begin{align}
 (\pi_h \xi)(p)&=\xi(p)\;\;\;\;\;\;\; \forall p\in\mathcal{V}_h,\label{interpo:1}\\
 \int_{e}\pi_h\xi\,ds&=\int_{e}\xi \,ds\;\;\;\;\; \forall e \in\mathcal{E}_h,\label{interpo:11}\\
 \int_{e}\frac{\partial (\pi_h\xi)}{\partial n}\,ds&=\int_{e}\frac{\partial \xi}{\partial n}\,ds\;\;\;\forall e\in \mathcal{E}_h. \label{interpo:2}
 \end{align}
We have the following interpolation error estimates
\begin{eqnarray}
\sum_{k=0}^2 h_T^k|\xi-\pi_h\xi|_{H^k(T)}\leq C h^{\gamma+2}|\xi|_{2+\gamma,2,T}\label{inter:err:es}
\end{eqnarray}
Use of \ref{interpo:11} gives
\begin{align*}
\int_{T}\nabla (\pi_h\xi)\,dx&=-\int_{T}0\,dx+\int_{\partial T}(\pi_h \xi)\,ds\notag\\
&=\sum_{e\in\mathcal{E}_T}\int_{e}(\pi_h \xi)\,ds=\sum_{e\in\mathcal{E}_T}\int_{e} \xi\,ds=\int_T\nabla \xi\,dx,
\end{align*} 
which implies
\begin{eqnarray*}
\Big|\frac{1}{|T|}\int_{T}\nabla (\pi_h\xi)\,dx\Big|=\Big|\frac{1}{|T|}\int_{T}\nabla \xi\,dx\Big|\leq \frac{1}{|T|}\int_{T}|\nabla\xi|\,dx\leq \frac{1}{|T|}|T|\leq 1.
\end{eqnarray*}
An application of  (\ref{interpo:2}) together with integration by parts leads to
 \begin{eqnarray}
 \int_{T}\Delta(\pi_h\xi)\,dx=\sum_{e\in \mathcal{E}_T}\int_{e}\frac{\partial (\pi_h\xi)}{\partial n}\,ds=\int_{T}(\Delta \xi)\,dx\;\;\;\;\;\forall T\in \mathcal{T}_h,
 \end{eqnarray}
  and hence
 \begin{eqnarray}
 Q_h(\Delta_h(\pi_h\xi))=Q_h(\Delta \xi)\;\;\;\;\;\forall \xi\in H^2(\Omega)\cap H^1_0(\Omega).\label{inter:err:3}
 \end{eqnarray}
 Use of   (\ref{intro:k}), (\ref{kh}), (\ref{interpo:1}) and (\ref{inter:err:3}) gives 
 \begin{eqnarray}
 \pi_h \mathcal{K}\subset \mathcal{K}_h.\label{k:kh}
 \end{eqnarray}
 Using the properties (\ref{continuous}), (\ref{coercive}) and (\ref{k:kh}), the solution of the discrete problem  discrete problem (\ref{dis:prob})  admits  a unique solution in the sense that the discrete variational inequality \cite{Glowinsiki:1984:book}
\begin{eqnarray}
\mathcal{A}_h({y}^*_h,v_h-{y}^*_h)\geq (y_d,v_h-{y}^*_h),\;\;\;\;\;\forall v_h\in \mathcal{K}_h.\label{discrete:variational:inequality}
\end{eqnarray} 
 \noindent
\textbf{The enriching operator} $E_h.$  The  enriching operator  $E_h$ is defined as follows:
\begin{eqnarray*}
 E_h:V_h\rightarrow M_h\subset H^2(\Omega)\cap H^1_0(\Omega),
 \end{eqnarray*}
  where $M_h$ denote the Hsieh-Clough-Tocher macro element space associated with $\mathcal{T}_h$. 
 The operator $E_h$ satisfies the following 
\begin{align}
(E_h v_h)(p)&=v_h(p)\;\;\;\;\;\;\;\;\;\forall p\in\mathcal{V}_h,\label{enriching:1}\\
\int_{e}\frac{\partial (E_hv_h)}{\partial n}\,ds&=\int_{e}\frac{\partial v_h}{\partial n}\,ds\;\;\;\; \forall e \in\mathcal{E}_h,\label{enriching:2}
\end{align}
We have some standard estimates
\begin{align}
\sum_{k=0}^2h^{2k}\sum_{T\in\mathcal{T}_h}|v_h-E_h v_h|^2_{H^k(T)}&\leq Ch^4\sum_{T\in\mathcal{T}_h}|v_h|^2_{H^2(T)}\;\;\;\;\;\forall v_h\in V_h,\label{enriching:3}\\
\sum_{k=0}^{2}h^k|\xi-E_h\pi_h\xi|_{H^k(\Omega)} &\leq Ch^{2+\gamma}|\xi|_{H^{2+\gamma}(\Omega)}\;\;\;\forall \xi\in H^{2+\gamma}(\Omega).\label{enriching:4}
\end{align} 
 We have the following relation for continuous and discrete bilinear form
\begin{eqnarray}
|\mathcal{A}_h(\pi_h\xi,v_h)-\mathcal{A}(\xi,E_h v_h)|\leq Ch^{\gamma}|\xi|_{H^{2+\gamma}(\Omega)}\|v_h\|_h\label{enriching:5}
\end{eqnarray}
for all $\xi\in H^{2+\gamma}(\Omega)$ and $v_h\in V_h$.
\noindent
Observe that  use of (\ref{enriching:2}) gives  the following analog of (\ref{inter:err:3}):
\begin{eqnarray}
Q_h(\Delta E_h v_h)=Q_h(\Delta_h v_h)\;\;\;\;\;\forall   v_h\in V_h.\label{enriching:6}
\end{eqnarray}

In the following lemma, we derive some error estimates of the enriching operator \cite{brenner:2013:cam,brenner:sung:2017}.
\begin{lemma}\label{lemma:1:1}
 For $v_h\in V_h$ and  $\xi\in H^{2+\gamma}(\Omega)$, there exists positive constant $C$, such that
\begin{eqnarray}
|\mathcal{A}_h(\xi,v_h-E_hv_h)|\leq Ch^{\gamma}\|\xi\|_{H^{2+\gamma}(\Omega)}\|v\|_{h}\;\;\;\;\;\forall \xi\in H^{2+\gamma}(\Omega).
\end{eqnarray}
\end{lemma}
\begin{proof}
For any $v_h\in V_h$, we have, by (\ref{enriching:1}),
\begin{eqnarray}
|\mathcal{A}_h(\xi,v_h-E_hv_h)|\leq \|\xi\|_{h}\|v_h-E_hv_h\|_{h}\leq C\|\xi\|_{H^2(\Omega)}\|v_h\|_{h}\;\;\;\;\forall \xi\in H^2(\Omega).\label{eq:1:1:1}
\end{eqnarray}
Let $\xi\in H^3(\Omega)$, $v_h\in V_h$ and $E_hv_h \in \mathcal{C}(\bar{\Omega})\cap H^1_0(\Omega)$, we have
\begin{align}
\mathcal{A}_h(\xi,v_h-E_hv_h)&=\sum_{T\in\mathcal{T}_h}\Big\{\int_T D^2\xi:D^2(v_h-E_hv_h)\,dx+\int_{T}\xi (v_h-E_hv_h)\,dx\Big\}\nonumber\\
&=-\sum_{T\in\mathcal{T}_h}\int_{T}\nabla (\Delta \xi)\cdot \nabla (v_h-E_hv_h)\,dx+\sum_{T\in\mathcal{T}_h}\int_{\partial T}D^2\xi:[\nabla (v_h-E_h v_h)\otimes n_T]\,ds\nonumber\\
&~~~~~+\sum_{T\in\mathcal{T}_h}\int_{T}\xi (v_h-E_hv_h)\,dx\nonumber
\\
&=-\sum_{T\in\mathcal{T}_h}\int_{T}\nabla (\Delta \xi)\cdot \nabla (v_h-E_hv_h)\,dx+\sum_{e\in\mathcal{E}_h}\int_{e}[D^2\xi-\overline{(D^2\xi)}_{e}]:[[\nabla (v-E_h v)\otimes n]]_{e}\,ds\nonumber\\
&~~~~~+\sum_{T\in\mathcal{T}_h}\int_{T}\xi (v_h-E_hv_h)\,dx\nonumber\\
&=: E_1+E_2+E_3,
\label{eq:1:1}
\end{align}
where $\overline{(D^2\xi)}_{e}$ is the average of $D^2\xi$ along $e$ and $[[\nabla (v-E_h v)\otimes n]]_{e}$ over the triangles that share a common edge $e$.
Now, we estimate $E_{i}|_{i=1,2,3}$ separately. 
The term $E_1$ is estimated as follows
\begin{align}
E_1\leq \Big|\sum_{T\in\mathcal{T}_h}\int_{T}\nabla (\Delta \xi)\cdot \nabla (v_h-E_hv_h)\,dx\Big|&\leq |\Delta \xi|_{H^1(\Omega)}\Big(\sum_{T\in\mathcal{T}_h}|v_h-E_h v_h|_{H^1(T)}^2\Big)^{\frac{1}{2}}\notag\\
&\leq Ch|\xi|_{H^3(\Omega)}\|v_h\|_{h}.\label{eq:1:2}
\end{align}
An application of the Cauchy-Schwartz inequality together with (\ref{enriching:3}) yields
\begin{eqnarray}
&&E_2\leq \Big|\sum_{e\in\mathcal{E}_h}\int_{e}[D^2\xi-\overline{(D^2\xi)}_{e}]:[[\nabla (v_h-E_h v_h)\otimes n]]_{e}\,ds\Big|\nonumber\\&&\leq \Big(\sum_{e\in\mathcal{E}_h}|e|^{-1}\|D^2\xi-\overline{(D^2\xi)}_{e}\|^2_{L^2(e)}\Big)^{\frac{1}{2}}\Big(\sum_{e\in \mathcal{E}_h}|e|\|[[\nabla (v_h-E_h v_h)\otimes n]]_{e}\|^2_{L^2(e)}\Big)^{\frac{1}{2}}\nonumber\\
&\leq&Ch|\xi|_{H^3(\Omega)}\|v_h\|_{h}.\label{eq:1:3}
\end{eqnarray}
 Use of the Cauchy-Schwarz inequality and (\ref{enriching:4}) yields
 \begin{eqnarray}
 E_3\leq Ch^2\|v_h\|_h\|\xi\|_{L^2(\Omega)}.\label{eq:1:3:1}
 \end{eqnarray}
Combining (\ref{eq:1:2}) and (\ref{eq:1:3:1}), we get
\begin{eqnarray}
|\mathcal{A}_h(\xi,v_h-E_hv_h)|\leq Ch|\xi|_{H^3(\Omega)}\|v_h\|_h\;\;\forall \xi\in H^3(\Omega).\label{eq:1:4}
\end{eqnarray}
It follows from the method of interpolation $H^{2+\gamma}(\Omega)=[H^2(\Omega),H^3(\Omega)]_{2,\gamma}$, (\ref{eq:1:1:1}) and (\ref{eq:1:4}) that
\begin{eqnarray*}
|\mathcal{A}_h(\xi,v_h-E_hv_h)|\leq Ch^{\gamma}\|\xi\|_{H^{2+\gamma}(\Omega)}\|v_h\|_{h}\;\;\;\;\;\forall \xi\in H^{2+\gamma}(\Omega),
\end{eqnarray*}
and hence we get the desired result.
\end{proof}
We have the following lemma which will be useful in our analysis.
\begin{lemma}\label{enriching:5:lemma}
For $v_h\in V_h$, and $\xi\in H^{2+\gamma}(\Omega)$, there exists a positive constant $C$ such that  the following 
\begin{eqnarray*}
|\mathcal{A}_h(\pi_h\xi,v_h)-\mathcal{A}(\xi,E_h v_h)|\leq Ch^{\gamma}|\xi|_{H{2+\gamma}(\Omega)}\|v_h\|_h
\end{eqnarray*}
holds. 
\end{lemma}
\begin{proof}
Adding and subtracting the term $\mathcal{A}_h(\xi,v_h)$, we get
\begin{align*}
|\mathcal{A}_h(\pi_h \xi,v_h)-\mathcal{A}(\xi,E_h v_h)|&\leq|\mathcal{A}_h(\pi_h\xi-\xi,v_h)|+|\mathcal{A}_h(\xi,v_h-E_h v_h)|\\
&\leq \Big(\sum_{T\in\mathcal{T}_h}|\pi_h \xi-\xi|^2_{H^2(T)}\Big)^{\frac{1}{2}}\Big(\sum_{T\in\mathcal{T}_h}|v|^2_{H^2(T)}\Big)^{\frac{1}{2}}+|\mathcal{A}_h(\xi,v_h-E_h v_h)|
\end{align*}
Using Lemma \ref{lemma:1:1} and (\ref{inter:err:es}), we obtain
\begin{eqnarray}
|\mathcal{A}_h(\pi_h \xi,v_h)-\mathcal{A}(\xi,E_h v_h)|\leq Ch^{\gamma}\|\xi\|_{H^{2+\gamma}(\Omega)}\|v_h\|_h\;\;\forall \xi\in H^{2+\gamma}(\Omega)\cap H^1_0(\Omega),\;\;v_h\in V_h.
\end{eqnarray}
This completes the proof.
\end{proof}
The following lemma 
follows by the idea of \cite[Appendix A.3]{kamana16}. 
\begin{lemma}
\begin{eqnarray}
\sum_{k=0}^2h^k|\xi-E_h\pi_h\xi|_{H^k(\Omega)}\leq Ch^{2+\gamma}|\xi|_{H^{2+\gamma}(\Omega)},
\end{eqnarray}
where $0\leq\gamma\leq 1$
and 
\begin{eqnarray}
\|\xi-E_h\pi_h\xi\|_{W^{1,\frac{2}{1-\delta}}}\leq C
\end{eqnarray}
for some $\delta>0$.
\end{lemma}

\section{Error Estimates}
In this section, we derive the convergence property of the state variable in $H^2$-type norm.
\begin{theorem}\label{thm:Apriori}
Let $y^*$ and $y_h^*$ be the solutions of (\ref{variational:ineq}) and (\ref{discrete:variational:inequality}), respectively. Then, we have
\begin{eqnarray}
\|{y}^*-{y}^*_h\|_{h}\leq Ch^{\tau},
\end{eqnarray}
where $\tau=min\{\gamma-\delta,1-\gamma-\delta,\gamma\}$, $C$ is a positive constant independent of mesh parameter $h$. 
\end{theorem}
\begin{proof}
To estimate the error $\|{y}^*-{y}^*_h\|_{h}$. We split the error as
\begin{eqnarray}
\|{y}^*-{y}^*_h\|_{h}\leq \|{y}^*-\pi_h{y}^*\|_h+\|\pi_h{y}^*-{y}^*_h\|_h.\label{split:1}
\end{eqnarray}
Use of (\ref{inter:err:es}) yields the bound the first term of the above equation. Now we estimate the second term as follows: An application  of (\ref{coercive}) yields 
\begin{align}\label{eq1}
 C \|\pi_h {y}^*-y^*_h\|_h^2 &\leq \mathcal{A}_h(\pi_h{y}^*-{y}^*_h,\pi_h{y}^*-{y}^*_h)\nonumber\\
 &\leq \mathcal{A}_h(\pi_h{y}^*,\pi_h{y}^*-{y}^*_h)-(y_d,\pi_h{y}^*-{y}^*_h) \nonumber\\
 & = \mathcal{A}_h(\pi_h{y}^*-{y}^*,\pi_h{y}^*-{y}^*_h)+ \mathcal{A}_h({y}^*,\pi_h{y}^*-{y}^*_h)-(y_d,\pi_h{y}^*-{y}^*_h) \nonumber\\
 &= \mathcal{A}_h(\pi_h{y}^*-{y}^*,\pi_h{y}^*-{y}^*_h) +\mathcal{A}_h({y}^*,E_h(\pi_h{y}^*-{y}^*_h))-(y_d,E_h(\pi_h{y}^*-{y}^*_h))
  \nonumber\\
 &~~~~+ \mathcal{A}_h({y}^*,\pi_h{y}^*-{y}^*_h-E_h(\pi_h{y}^*-{y}^*_h))-(y_d,\pi_h{y}^*-{y}^*_h-E_h(\pi_h{y}^*-{y}^*_h)). 
 \end{align}
Now we estimate the terms of the right hand side of \eqref{eq1} as follows:  an use of (\ref{continuous}) together with  (\ref{inter:err:es}) gives
  \begin{align}\label{EstT1}
  \mathcal{A}_h(\pi_h{y}^*-{y}^*,\pi_h{y}^*-{y}^*_h) &\leq  C_1\|\pi_h{y}^*-{y}^*\|_h\|\pi_h{y}^*-{y}^*_h\|_h \notag \\
  & \leq C h^{\gamma} |{y}^*|_{H^{2+\gamma}(\Omega)} \|\pi_h{y}^*-{y}^*_h\|_h.
  \end{align}
   Using the Cauchy Schwartz inequality,      (\ref{enriching:3}) and (\ref{enriching:5}) to obtain
\begin{align}\label{EstT2}
\mathcal{A}_h({y}^*,\pi_h{y}^*-{y}^*_h-E_h(\pi_h{y}^*-{y}^*_h))&-(y_d,\pi_h{y}^*-{y}^*_h-E_h(\pi_h{y}^*-{y}^*_h))\notag\\& \leq C  h^{\gamma} |{y}^*|_{H^{2+\gamma}(\Omega)} \|\pi_h{y}^*-{y}^*_h\|_h.
\end{align}  
 \noindent
Using the  Karush-Kuhn-Tucker condition (\ref{kkt}), we have 
  \begin{align}
  \mathcal{A}_h({y}^*,E_h(\pi_h{y}^*-{y}^*_h))-(y_d,E_h(\pi_h{y}^*-{y}^*_h)) &= \int_{\Omega}\lambda(-\Delta E_h(\pi_h{y}^*-{y}^*_h))~dx\nonumber\\&~~~~+\int_{\Omega} \nabla {y}^* \nabla E_h(\pi_h{y}^*-{y}^*_h)\,d\mu .\label{conv:2}
  \end{align}
  First we estimate  the term related to the control constraints of (\ref{conv:2}). 
  Use of the complementary conditions (\ref{lambda:1})-(\ref{lambda:3}) yields\begin{eqnarray}
\int_{\Omega}\lambda[-\Delta E_h(\pi_h{y}^*-{y}^*_h)]dx=\int_{\Omega}\lambda_1[-\Delta E_h(\pi_h{y}^*-{y}^*_h)]dx+\int_{\Omega}\lambda_2[-\Delta E_h(\pi_h{y}^*-{y}^*_h)]dx,\label{conv:1:2}
\end{eqnarray}
where 
\begin{eqnarray}
\lambda_1=
\begin{cases}
\lambda\;\;\;\;\;\text{if}\;\;-\Delta y=u_a\\
0\;\;\;\;\;\text{otherwise},
\end{cases}
\label{lambda:conv:1}
\end{eqnarray}
\begin{eqnarray}
\lambda_2=
\begin{cases}
\lambda\;\;\;\;\;\text{if}\;\;-\Delta y=u_b\\
0\;\;\;\;\;\text{otherwise},
\end{cases}\label{lambda:conv:2}
\end{eqnarray}
Here $\lambda_1=\max(\lambda,0)$ and $\lambda_2=\min(\lambda,0)$. Therefore,
 a standard interpolation error estimate  yields
\begin{eqnarray}
\|\lambda_j-Q_h\lambda_j\|_{L^2(\Omega)}\leq C\,h^{\gamma-\delta}|\lambda_j|_{W^{\gamma,\frac{2}{1+\delta}}(\Omega)}\;\;\;\;\;\text{for}\;\;j=1,2.\label{conv:3}
\end{eqnarray}
To estimate the first term of (\ref{conv:1:2}),   add and subtract the terms, to obtain 
\begin{align}
\int_{\Omega}\lambda_1[-\Delta E_h(\pi_h{y}^*-{y}^*_h)]dx&=\int_{\Omega}\lambda_1[-\Delta (E_h\pi_h{y}^*-{y}^*)]\,dx+\int_{\Omega}\lambda_1(u_a-Q_h u_a)\,dx\nonumber\\&+\int_{\Omega}\lambda_1Q_h(u_a+\Delta E_h{y}^*_h)\,dx+\int_{\Omega}\lambda_1[\Delta E_h{y}^*_h-Q_h(\Delta E_h{y}^*_h)]\,dx\nonumber\\
&=:\hat{E}_1+\hat{E}_2+\hat{E}_3+\hat{E}_4.\label{conv:4}
\end{align}
From (\ref{inter:err:3}) and  (\ref{enriching:6}), we have
 \begin{equation*}
 Q_h\Delta(E_h\pi_h{y}^*-{y}^*)=0,
 \end{equation*}
  Using (\ref{enriching:4}) and (\ref{conv:3}) to obtain
 \begin{align*}
\hat{E}_1&=\int_{\Omega}\lambda_1[-\Delta (E_h\pi_h{y}^*-{y}^*)]\,dx\nonumber\\&=\int_{\Omega}(\lambda_1-Q_h\lambda_1)[-\Delta (E_h\pi_h{y}^*-{y}^*)]\,dx\nonumber\\&\leq C\,h^{(2\gamma-\delta)}|\lambda_1|_{W^{\gamma,\frac{2}{1+\delta}}(\Omega)}|y^*|_{H^{2+\gamma}(\Omega)}.
 \end{align*}
 Similarly,
 \begin{align}
 \hat{E_2}&=\int_{\Omega}\lambda_1(u_a-Q_h u_a)\,dx\nonumber\\&=\int_{\Omega}(\lambda_1-Q_h\lambda_1)(u_a-Q_h u_a)\,dx\nonumber\\&\leq Ch^{(2+\gamma-\delta)}|\lambda_1|_{W^{\gamma,\frac{2}{1+\delta}}(\Omega)}\|u_a\|_{W^{1,\infty}(\Omega)}.\label{1:2::3}
 \end{align}
   An application of  (\ref{lambda:1}), (\ref{kh}), (\ref{enriching:6}) and (\ref{lambda:conv:1}) gives
 \begin{align*}
 \hat{E}_3&=\int_{\Omega}\lambda_1Q_h(u_a+\Delta E_h{y}^*_h)\,dx\nonumber\\&=\int_{\Omega}\lambda_1(Q_hu_a+Q_h(\Delta_h{y}^*_h))\,dx\leq 0,
 \end{align*}
The  term $\hat{E}_4$  can be written as 
\begin{align*}
\hat{E}_4&=\int_{\Omega}\lambda_1[\Delta E_h{y}^*_h-Q_h(\Delta E_h{y}^*_h)]\,dx\\&=\int_{\Omega}\lambda_1[\Delta(E_h{y}^*_h-{y}^*)-Q_h\Delta (E_h{y}^*_h-{y}^*)]\,dx+\int_{\Omega}\lambda_1(\Delta {y}^*-Q_h\Delta {y}^*)\,dx,
\end{align*}
and we have
\begin{align*}
\int_{\Omega}\lambda_1(\Delta {y}^*-Q_h\Delta {y}^*)\,dx&=\int_{\Omega}(\lambda_1-Q_h\lambda_1)(\Delta {y}^*-Q_h\Delta{y}^*)\,dx\nonumber\\&\leq Ch^{2\gamma-\delta}|\lambda_1|_{W^{\gamma,\frac{2}{1+\delta}}(\Omega)}|\Delta{y}^*|_{H^{\gamma}(\Omega)},\label{conv:8}
\end{align*}
since $\Delta{y}^*\in H^{\gamma}(\Omega)$. Finally, (\ref{enriching:4}) and (\ref{conv:3}) we have
\begin{align}
&\int_{\Omega}\lambda_1[\Delta(E_h{y}^*_h-{y}^*)-Q_h\Delta (E_h{y}^*_h-{y}^*)]\,dx=\int_{\Omega}(\lambda_1-Q_h\lambda_1)[\Delta (E_h{y}^*_h-{y}^*)]\,dx\nonumber\\
&\leq Ch^{\gamma-\delta}|\lambda_1|_{W^{\gamma,\frac{2}{1+\delta}}(\Omega)}\Big(|E_h({y}^*_h-\pi_h{y}^*)|_{H^2(\Omega)}+|E_h\pi_h{y}^*-{y}^*|_{H^2(\Omega)}\Big)\nonumber\\
&\leq Ch^{\gamma-\delta}|\lambda_1|_{W^{\gamma,\frac{2}{1+\delta}}(\Omega)}\,\Big(\|\pi_h{y}^*-{y}^*_h\|_h+h^{\gamma}|{y}^*|_{H^{2+\gamma}(\Omega)}\Big).
\end{align}
Thus, 
\begin{align*}
\hat{E}_4\leq Ch^{\gamma-\delta}|\lambda_1|_{W^{\gamma,\frac{2}{1+\delta}}(\Omega)}\,\Big(\|\pi_h{y}^*-{y}^*_h\|_h+h^{\gamma}|{y}^*|_{H^{2+\gamma}(\Omega)}\Big)+
Ch^{2\gamma-\delta}|\lambda_1|_{W^{\gamma,\frac{2}{1+\delta}}(\Omega)}|\Delta {y}^*|_{H^{\gamma}(\Omega)}.
\end{align*}
Combining the estimates $\hat{E}_i|_{i=1,2,3,4}$ to have
\begin{eqnarray}
\int_{\Omega}\lambda_1[-\Delta E_h(\pi_h{y}^*-{y}^*_h)]\,dx\leq \hat{C}h^{\gamma-\delta}\,|\lambda_1|_{W^{\gamma,\frac{2}{1+\delta}}(\Omega)}\Big(\|\pi_h{y}^*-{y}^*_h\|_h+h^{\gamma}|{y}^*|_{H^{2+\gamma}(\Omega)}\Big).\label{conv:10}
\end{eqnarray}
Similarly, we obtain 
\begin{eqnarray}
\int_{\Omega}\lambda_2[-\Delta E_h(\pi_h{y}^*-{y}^*_h)]\,dx\leq \hat{C}h^{\gamma-\delta}|\lambda_2|_{W^{\gamma,\frac{2}{1+\delta}}(\Omega)}\,\Big(\|\pi_h{y}^*-{y}^*_h\|_h+h^{\gamma}|{y}^*|_{W^{2+\gamma,2}(\Omega)}\Big).\label{conv:11}
\end{eqnarray}
Altogether (\ref{conv:1:2}), (\ref{1:2::3}), (\ref{conv:10})  and (\ref{conv:11}) gives
\begin{align}
\int_{\Omega}\lambda[-\Delta E_h(\pi_h{y}^*-{y}^*_h)]\,dx\leq \hat{C}h^{\gamma-\delta}\Big(|\lambda_1|_{W^{\gamma,\frac{2}{1+\delta}}(\Omega)}+|\lambda_2|_{W^{\gamma,\frac{2}{1+\delta}}(\Omega)}\Big)\,\Big(\|\pi_h{y}^*-{y}^*_h\|_h+h^{\gamma}|{y}^*|_{H^{2+\gamma}(\Omega)}\Big).\label{conv:12}
\end{align}
 We estimate the second  term of (\ref{conv:2}) related to the gradient state constraint 
  \begin{align*}
  \int_{\Omega}\nabla{y}^*\nabla E_h(\pi_h{y}^*-{y}^*_h)\,d\mu&=\int_{\Omega}\nabla{y}^* (\nabla E_h\pi_h{y}^*-\nabla E_h{y}^*_h)\,d\mu\\
  &= \int_{\Omega}\nabla{y}^*(\nabla E_h\pi_h{y}^*-\nabla{y}^*)\,d\mu+\int_{\Omega}\nabla{y}^*(\nabla {y}^*-I_h\nabla {y}^*_h)\,d\mu\\&~~~~+\int_{\Omega}\nabla {y}^*(I_h\nabla{y}^*_h-I_h\nabla E_h{y}^*_h)\,d\mu+\int_{\Omega}\nabla\bar{y}(I_h\nabla E_h\bar{y}_h-\nabla E_h\bar{y}_h)\,d\mu\\
& =:G_1+G_2+G_3+G_4.
  \end{align*}
We estimate each term separately. For $G_1$, we have 
\begin{align}
G_1&=\int_{\Omega}\nabla{y}^*(\nabla E_h\pi_h{y}^*-\nabla{y}^*)\,d\mu\notag\\
&\leq  C\|\nabla E_h\pi_h{y}^*-\nabla{y}^*\|_{W^{\gamma,\frac{2}{1+\delta}}(\Omega)}\notag\\
&\leq   Ch^{1+\delta}\|{y}^*\|_{H^{2+\gamma}(\Omega)}.\label{g1:1}
\end{align}
Use of complementary condition (\ref{comple:1}) yields
  \begin{eqnarray}
  G_2=\int_{\Omega}\nabla{y}^*(\nabla{y}^*-I_h\nabla {y}^*_h)\,d\mu\leq 0,\label{g2}
  \end{eqnarray}
  and the term $(I_h\nabla {y}^*_h-I_h\nabla E_h{y}^*_h)=0$ by the property of the enriching map $E_h$ gives
  \begin{align}
  G_3=\int_{\Omega}\nabla {y}^*(I_h\nabla{y}^*_h-I_h\nabla E_h{y}^*_h)\,d\mu=0.\label{g3}
  \end{align}
 For $G_4$, we  split the term as follows 
\begin{align*}
G_4&=\int_{\Omega}\nabla{y}^*(I_h\nabla E_h{y}^*_h-\nabla E_h{y}^*_h)\,d\mu\\
&=\int_{\Omega}\nabla{y}^*(I_h(\nabla E_h{y}^*_h-\nabla {y}^*)-(\nabla E_h{y}^*_h-\nabla{y}^*))\,d\mu+\int_{\Omega}\nabla{y}^*(I_h\nabla{y}^*-\nabla{y}^*)\,d\mu\\
&=:\hat{G}_1+\hat{G}_2.
\end{align*}
Using the  regularity of $\mu$ 
\begin{align}
\hat{G}_1&=\int_{\Omega}\nabla{y}^*(I_h(\nabla E_h{y}^*_h-\nabla {y}^*)-(\nabla E_h{y}^*_h-\nabla {y}^*))\,d\mu\notag\\
&\leq C\|I_h(\nabla E_h{y}^*_h-\nabla {y}^*)-(\nabla E_h{y}^*_h-\nabla {y}^*)\|_{W^{\gamma,\frac{2}{1-\delta}}(\Omega)}\notag\\
&\leq Ch^{2(\frac{1-\delta}{2}-\frac{1}{2})}h^{1-\gamma}| E_h{y}^*_h- {y}^*|_{H^2(\Omega)}\notag\\
&\leq Ch^{1-\delta-\gamma}| E_h{y}^*_h-{y}^*|_{H^2(\Omega)}\notag\\
&\leq  C h^{1-\delta-\gamma}\Big[|E_h({y}^*_h-\pi_h{y}^*)|_{H^2(\Omega)}+|E_h\pi_h{y}^*-{y}^*|_{H^2(\Omega)}\Big]\notag\\
&\leq  Ch^{1-\delta-\gamma}[\|{y}^*_h-\pi_h{y}^*\|_{h}+h^{\gamma}|{y}^*|_{W^{2+\gamma,2}(\Omega)}].
\end{align}
Note that
\begin{align*}
\hat{G}_2&=\int_{\Omega}\nabla{y}^*(I_h\nabla{y}^*-\nabla{y}^*)\,d\mu\\&\leq \|I_h \nabla{y}^*-\nabla{y}^*\|_{L^{\infty}(\Omega)}\\
&\leq Ch^{\gamma}|{y}^*|_{H^{2+\gamma}(\Omega)}.
\end{align*}
Combine the estimate of $\hat{G}_1$ and $\hat{G}_2$ gives
\begin{align}
G_4\leq Ch^{1-\delta-\gamma}\Big[\|{y}^*_h-\pi_h{y}^*\|_{h}+h^{\gamma}|{y}^*|_{H^{2+\gamma}(\Omega)}\Big]+
Ch^{\gamma}|{y}^*|_{H^{2+\gamma}(\Omega)}.\label{g4}
\end{align} 
Altogether (\ref{g1:1})-(\ref{g4}) and use the fact that $\gamma<1$,we obtain 
\begin{align}
\int_{\Omega}\nabla{y}^*\nabla E_h(\pi_h{y}^*-{y}^*_h)\,d\mu &\leq Ch^{1+\delta}|{y}^*|_{H^{2+\gamma}(\Omega)}+Ch^{1-\delta-\gamma}\Big[\|{y}^*_h-\pi_h{y}^*\|_{h}+h^{\gamma}|{y}^*|_{H^{2+\gamma}(\Omega)}\Big]\notag\\&+
Ch^{\gamma}|{y}^*|_{W^{2+\gamma,2}(\Omega)}\notag\\
&\leq Ch^{\gamma} |{y}^*|_{H^{2+\gamma}(\Omega)}+Ch^{1-\delta-\gamma}\|{y}^*_h-\pi_h{y}^*\|_{h}.\label{state:cons:es}
\end{align}
An application of (\ref{eq1}), (\ref{EstT1}), (\ref{EstT2}), (\ref{conv:2}), (\ref{conv:12}) and (\ref{state:cons:es}) together with the Young's
inequality leads to
\begin{eqnarray*}
\|\pi_h {y}^*-y^*_h\|_h\leq Ch^{\min\{\gamma-\delta,1-\gamma-\delta,\gamma\}}.\label{final:mid:1}
\end{eqnarray*}
Estimate (\ref{final:mid:1}) and (\ref{split:1}) completes the rest  of the proof.
   \end{proof}
    By taking the piece-wise linear function $\bar{u}_h=-\Delta_h\bar{y}_h$ as an approximation  of the optimal control $\bar{u}$, The following estimate is the direct consequence of Theorem \ref{thm:Apriori}.
 \begin{corollary} \label{cor:Lowu}
  There exists a positive constant $C$ independent of $h$ such that
  \begin{eqnarray*}
  \|\bar{u}-\bar{u}_h\|_{L^2(\Omega)}\leq \hat{C}\,h^{\tau},
  \end{eqnarray*}
 where $\tau$ is defined in Theorem \ref{thm:Apriori}and $\hat{C}=C\Big(|\bar{y}|_{H^{2+\alpha}(\Omega)},\;\|y_d\|_{L^2(\Omega)}, \;|\mu|\Big)$.
 \end{corollary}
   \begin{remark}
   In  \cite{wollner11}, Remark 1, shows that $\gamma=(1-\frac{2}{r}-\epsilon)$ and $\gamma'={1}-\frac{2}{r}$ for any $\epsilon>0$. Using the Theorem \ref{thm:Apriori}, we observe that
   $\tau=min\{\gamma-\delta,1-\gamma-\delta,\gamma\}$. 
   \begin{enumerate}
   \item $\gamma-\delta=1-\frac{2}{r}-\epsilon'$,
   \item $1-\gamma-\delta=1-(1-\frac{2}{r}-\epsilon)-\delta$=$\frac{2}{r}+\epsilon-\delta\geq \frac{2}{r}$
   \end{enumerate}
   \end{remark}
 \begin{remark}\label{rem:4}
  Consider the general minimization problem as follows: 
 \begin{eqnarray}
\text{find}\;\; (\bar{y},\bar{u})= \displaystyle argmin_{(y,u)\in \mathcal{K}} \Big(\frac{1}{2}\|y-y_d\|^2 +\frac{\beta}{2}\|u\|^2  \Big)\label{intro:functional2}
\end{eqnarray}
 subject to
\begin{eqnarray}
\begin{cases}
\int_{\Omega}\nabla y \cdot \nabla v~dx = \int_{\Omega} (u+f) v~dx,\;\;\;\;\forall v \in H^1_0(\Omega),\\
|\nabla y|\leq 1\;\;\;\;\; \\
u_a\leq u \leq u_b,
\end{cases}\label{intro:state:cons2}
\end{eqnarray} 
where $f \in L^2(\Omega)$ is a given function. It is easy to see that the results obtained in this paper can be extended to the above problem.
 \end{remark}

  \section{Numerical Experiment}
  In this section, we demonstrate our theoretical findings on two-dimensional examples. In the presence of gradient state constraint, the discrete problem becomes nonlinear. To solve the optimization problem the primal-dual active set method \cite{BIK:1999:PrimalDual,kunisch2002,kunisch2003,ito} is not applicable. We have used the Uzawa algorithm \cite{ciarlet,Glowinsiki:1984:book} to solve the discrete problem.  All the computations are done using MATLAB software. 
  \begin{exam}\label{ex:1}
  In this example, we have consider the optimization problem (\ref{intro:functional2})-(\ref{intro:state:cons2}). 
  The data of the example is taken from \cite{schila11}, Consider $\Omega$ be the circle centered at origin and of radius $2$, $\beta=1$,
  \begin{eqnarray*}
  y_{d}(x)&:=&\begin{cases}
  \frac{1}{4}+\frac{1}{2}log \,2-\frac{1}{4}|x|^2,\;\;\;\;0\leq |x|\leq  1,\\
  \frac{1}{2}log\,2-\frac{1}{2} log|x|,\;\;\;\;\;\;\;\;1< |x|\leq 2.
  \end{cases}\\
  f(x)&:=&\begin{cases} 2,\;\;\;\;\;0\leq |x|\leq 1,\\
  0,\;\;\;\;\;1<|x|\leq 2.
 \end{cases}\\
  u(x)&=&\begin{cases}-1,\;\;\;\;\;0\leq|x|\leq 1,\\
  0,\;\;\;\;\;1<|x|\leq 2
  \end{cases}
  \end{eqnarray*}
  with corresponding state $y=y_d$, the control bounds $u_a=-2$ and $u_b=2$. 
 \end{exam}   
  The discrete problem is nonlinear due to the presence of the gradient state constraint. So, we have used the Uzawa method to solve the discrete problem. The errors and order of convergence in different norms are presented in Table $1$.  We observe that the order of convergence for the in the $H^2$-norm is matched with the order of convergence for the control variable in the $L^2$-norm in \cite{wollner11}. The profiles of the exact and approximate state are presented in Figure \ref{figure}.
   
 \begin{table}[h!]
\begin{center}
\caption{ Numerical results of example \ref{ex:1}.}
\begin{tabular}{|c|c|c|c|c|c|c|}
\hline
h & $\|y-y_h\|_{h}$ &order&$\|y-y_h\|_{H^1(\Omega)}$ &order & $\|y-y_h\|_{L^2(\Omega)}$&order \\
\hline
$2^{-2}$  &  $1.4519\times 10^{0}$  &    - & $0.8118\times 10^{0}$ &      -   & $0.4186\times 10^{0}$ &  -\\
$2^{-3}$  &  $1.2474\times 10^{0}$  &0.2644& $0.3407\times 10^{0}$ &  1.5117  & $0.1617\times 10^{0}$  &  1.6559\\
$2^{-4}$  &  $0.7906\times 10^{0}$  &0.7220& $0.1600\times 10^{0}$ &  1.1967  & $0.1156\times 10^{0}$  &  0.5321\\
$2^{-5}$  &  $0.5975\times 10^{0}$ &0.4230& $0.0732\times 10^{0}$ &  1.1805  & $0.0539\times 10^{0}$ &  1.1511    \\
$2^{-6}$  &  $0.3878\times 10^{0}$ &0.6380& $0.0329\times 10^{0}$ &  1.1827  & $0.0250\times 10^{0}$ &  1.1323 \\
\hline
\end{tabular}
\end{center}
\end{table}
   
  \begin{figure}[h!]
\centering
  \includegraphics[width=12cm, height=5cm]{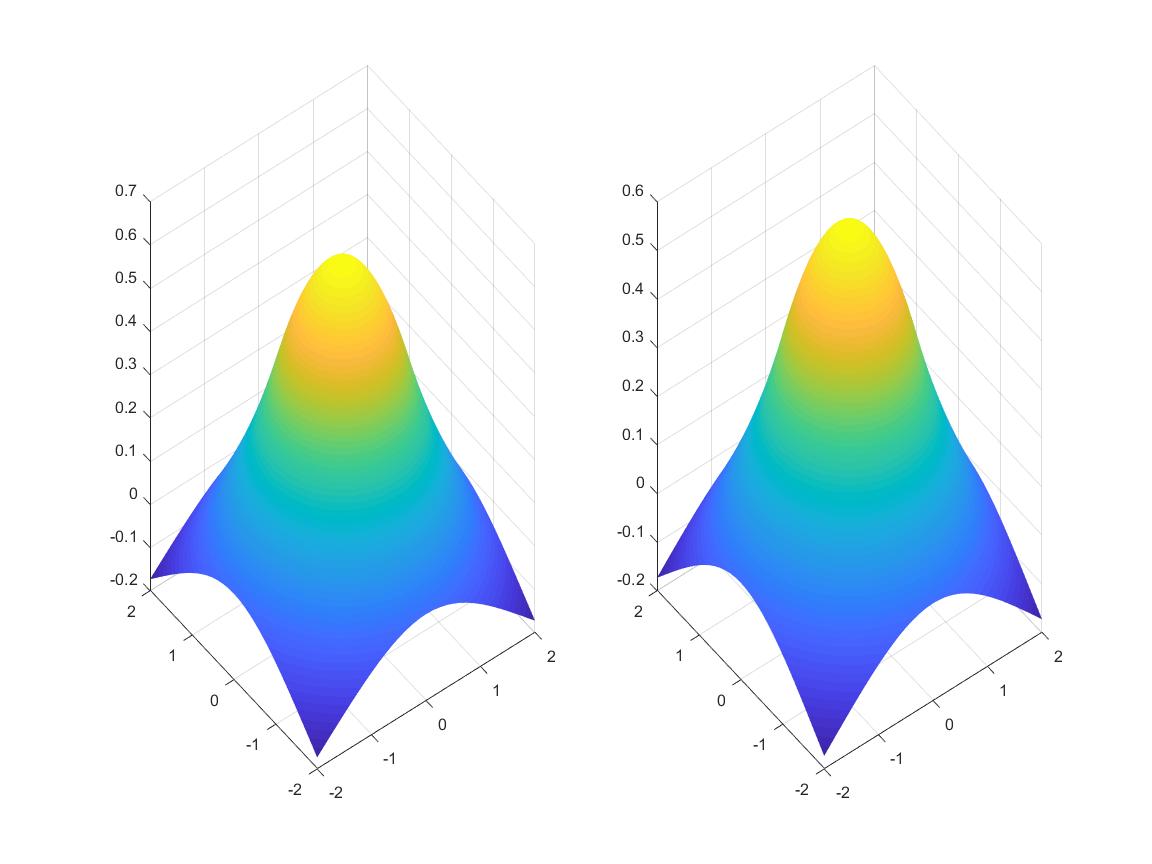}
\caption{The profile of the discrete state(left) and the exact state(right)  for Example \ref{ex:1} with $h=2^{-7}$.}\label{figure}
\end{figure}

\end{document}